\documentclass[preprint,12pt]{amsart}

\usepackage{amssymb}
\usepackage{amsmath}
\usepackage{amsthm}
\usepackage[foot]{amsaddr}

\usepackage{mathrsfs}
\usepackage{calrsfs} 
\usepackage{tikz}
\usetikzlibrary{calc}
\parskip\medskipamount
\newcommand{\K}{\mathbb{K}}

\renewcommand{\P}{\mathcal{P}}
\renewcommand{\L}{\mathcal{L}}
\newcommand{\V}{\mathcal{V}}

\newcommand{\I}{\mathtt{I}}
\newcommand{\C}{\mathtt{C}}

\newcommand{\Oct}{\mathscr{O}}
\newcommand{\MS}{\mathscr{M}}

\newcommand{\PO}{\P^\Oct}
\newcommand{\LO}{\L^\Oct}

\newcommand{\IO}{\I^\Oct}

\newcommand{\PMS}{\P^\MS}
\newcommand{\LMS}{\L^\MS}
\newcommand{\PiMS}{\mathcal{V}^\MS}
\newcommand{\SMS}{\mathcal{S}^\MS}

\newcommand{\CMS}{\C^\MS}

\newcommand{\OctFull}{\Oct = (\PO,\LO,\IO)}
\newcommand{\MSFull}{\MS = (\PMS,\LMS,\PiMS,\SMS,\CMS)}

\newcommand{\polarity}{\theta}

\newtheorem{theorem}{Theorem}[section]
\newtheorem{thmx}{Theorem}

\newtheorem{corollary}[theorem]{Corollary}
\newtheorem{lemma}[theorem]{Lemma}

\theoremstyle{remark}
\newtheorem{remark}{Remark}

\theoremstyle{definition}
\newtheorem{definition}[theorem]{Definition}

\newcommand{\col}{\perp}
\newcommand{\sympl}{\perp\hspace{-0.4em}\perp}
\renewcommand{\special}{\bowtie}
\newcommand{\opp}{\equiv}
\newcommand{\dual}{symmetric}
\newcommand{\Dual}{Symmetric}
\newcommand{\pperp}{\sympl}

\newcommand{\sympThrough}[2]{#1\diamond#2}
\newcommand{\specPoint}[2]{[#1,#2]}

\begin{document}

\title{The natural embedding of the Ree-Tits Octagon}

\author{Sebastian Petit}
\email{sebastian.petit@pg.canterbury.ac.nz}
\address{
	School of Mathematics and Statistics,
	University of Canterbury,
	Science road, 
	8041, Christchurch, 
	Canterbury, New Zealand}

\author{Hendrik Van Maldeghem}
\address{
	Department of Mathematics, Computer Science \& Statisctics, 
	Ghent University,
	Krijgslaan 299, S9,
	9000, Gent,
	Belgium}
	
\begin{abstract}
	In this paper, we study and characterise the natural embedding of the Ree-Tits octagons. 
	\\
	
	\\
	
\end{abstract}

\maketitle


\section{Introduction}	

In his seminal paper \cite{Tits:59}, Jacques Tits introduced geometric structures called \emph{generalised polygons}, and more specific \emph{generalised hexagons}, as the natural permutation representations for Dickson's simple groups (Chevalley groups of type $\mathsf{G_2}$), and the new simple groups of type $^3\mathsf{D_4}$ that he discovered in the same paper. These geometries were the predecessors of what later would be called \emph{buildings}, and, more precisely, their incidence graphs are, viewed as complexes, \emph{spherical buildings of rank $2$}. Generalised triangles and generalised quadrangles were known before as projective planes and polar geometries (called \emph{polar spaces} nowadays) of rank 2. They are related to classical groups. Apart from the exceptional groups studied and discovered in \cite{Tits:59}, there is one other class of simple groups whose members act naturally on generalised polygons and that is the class of Ree groups in characteristic $2$, which act on generalised octagons, called Ree-Tits octagons \cite{Tits:60}. 

In order to better understand these geometries, and hence also the corresponding classical and exceptional groups, one tries to find representations of them inside other, more accessible (higher rank) geometries. The natural homes of the generalised quadrangles that are related to classical groups are the projective spaces since these geometries are defined using forms on vector spaces. The generalised hexagons defined and discovered in \cite{Tits:59} are related to a triality map and hence can be represented in certain polar spaces, which, on their turn, naturally sit in projective spaces again. Finally, the Ree-Tits octagons, which are the most elusive examples of generalised polygons, appear as fix structure of a polarity in \emph{metasymplectic spaces}. The latter are geometries attached to buildings of type $\mathsf{F_4}$. They were studied by Freudenthal in the real case, and later by Tits \cite{Tits:74} in the general case. Cohen \cite{Cohen} formulated an alternative axiom system that is still used today (but we will use Tits' definition 10.13 of \cite{Tits:74}, see Definition~\ref{def:MSS}). 

The inclusion of the Ree-Tits octagons in metasymplectic spaces was used to provide an abstract characterisation of the former in \cite{Mal:98} in case the underlying fields is perfect, and to work out the combinatorics of root subgroups in groups of type $^2\mathsf{F_4}$ in \cite{Sar:88}.  

Metasymplectic spaces are in many respects very interesting geometries, as they come in many flavours, due to the existence of the appropriate forms of types $\mathsf{E_6,E_7}$ and $\mathsf{E_8}$ (see \cite{Tits:66}).  They contain representations of important classes of generalised quadrangles, as recently shown in \cite{Lam-Mal:26}, and also of generalised hexagons (see \cite{G2inF4}). Moreover, as pointed out above, also the Ree-Tits octagons have a representation, or embedding, in metasymplectic spaces.  That embedding was recently characterised \cite{Rijpert2025} as arising from a subcomplex of a building of type $\mathsf{F_4}$ possessing certain properties (which we recall in Section~\ref{Sec:ReeTits} below). However, that characterisation is not entirely satisfying in that the conditions are rather technical, self-dual and hence in a way too strong---it sort of presupposes a polarity, although the actual construction of the polarity is highly non-trivial.  In the present paper, we use this characterisation to prove an ostensibly stringer characterisation whose assumptions are much less technical, more natural, and are formulated in the language of geometries instead of simplical complexes. More exactly, we will work with three geometric conditions and variations, which we briefly review and motivate in the next paragraph.

We start off with a \emph{fully embedded} arbitrary generalised octagon $\Oct$ in a metasymplectic space $\MS$. This means that the (different) points of $\Oct$ are (different) points of $\MS$ and each line of $\Oct$, viewed as sets of points, coincides with a line of $\MS$. Our first basic assumption is that the lines of $\Oct$ through any point are contained in a symplecton (called a \emph{hyperline} in \cite{Tits:83}) through that point. This is the equivalent of what people call a \emph{polarised embedding} (of a polar space or a dual polar space): the set of lines through each point is contained in an object that has the type of the images of a point under the polarity (or triality). The second basic assumption is a kind of density assumption, in order to keep the metasymplectic space ``small'' enough compared to the octagon: it says that every symplecton must intersect the octagon non-trivially. A variation is, more specifically, that each symplecton intersects $\Oct$ in either a point, a line (of $\Oct$), or a all lines of $\Oct$ through a point. If we require that variation, then no additional---third--- assumption is needed. If we only require a non-trivial intersection, a third condition must ensure that distances measured in $\Oct$ compare in a reasonable way to those measured in $\MS$: either we ask that points at distance~$2$ in $\Oct$ are symplectic in $\MS$ (meaning, contained in a unique symplecton), or there exists a pair of points in $\Oct$ that is at maximal distance in $\MS$. We believe that these conditions are much easier to check ``in real life'' than the ones from \cite{Rijpert2025}.

Referring forward to notation, we show the following three variation of the same result.

\begin{thmx}\label{Main}
	Let \(\Oct\) 
	be a generalised octagon fully embedded in a metasymplectic space \(\MS\). 
	Then \(\Oct\) is a Ree-Tits octagon naturally associated to a polarity of \(\MS\) if, and only if, 
	\begin{itemize}
		\item[\((S)\)] two points of \(\Oct\) are at distance \(2\) if, and only if, they are symplectic,
		\item[\((P)\)] through every point \(x\) of \(\Oct\) there exists a symp \(x^\polarity\) of \(\MS\) such that all points collinear to \(x\) in \(\Oct\) are contained in \(x^\polarity\),
		\item[\((D)\)] every symp of \(\MS\) contains at least one point of \(\Oct\).
	\end{itemize}
\end{thmx}

\begin{thmx}\label{Var1}
	Let \(\Oct\) be a generalised octagon fully embedded in a metasymplectic space \(\MS\).
	Then \(\Oct\) is a Ree-Tits octagon naturally associated to a polarity of \(\MS\) if, and only if, 
	\begin{itemize}
		\item[\((P)\)] through every point \(x\) of \(\Oct\) there exists a symp \(x^\polarity\) of \(\MS\) such that all points collinear to \(x\) in \(\Oct\) are contained in \(x^\polarity\),
		\item[\((D^\star)\)] every symp of \(\MS\) intersects \(\Oct\) in either a point, a line of $\Oct$, or all points at distance~$1$ from a point of \(\Oct\).
	\end{itemize}
\end{thmx}

\begin{thmx}\label{Var2}
	Let \(\Oct\) be a generalised octagon fully embedded in a metasymplectic space \(\MS\).
	Then \(\Oct\) is a Ree-Tits octagon naturally associated to a polarity of \(\MS\) if, and only if, 
	\begin{itemize}
		\item[\((O)\)] there exist two points of \(\Oct\) that are opposite in \(\MS\),
		\item[\((P)\)] through every point \(x\) of \(\Oct\) there exists a symp \(x^\polarity\) of \(\MS\) such that all points collinear to \(x\) in \(\Oct\) are contained in \(x^\polarity\),
		\item[\((D)\)] every symp of \(\MS\) contains at least one point of \(\Oct\).
	\end{itemize}
\end{thmx}

Note that our results not only provide characterisations of the natural embeddings of the Ree-Tits octagons, but also of the Ree-Tits octagons themselves.  However, besides free constructions and variations of these, no other generalised octagons are known; in particular the only known finite  generalised octagons are the Ree-Tits octagons over finite fields of characteristic~$2$.

\textbf{Structure of the paper} \\
After a brief introduction (Section~\ref{Sec:Background}), we show our main results using a number of steps, which independently prove various equivalences of conditions on embedded generalised octagons in metasymplectic spaces.

First (Section~\ref{Sec:PseudoIsoEmbedding}, Theorem~\ref{Thm:PseudoIsoEmbedding}), we investigate \emph{pseudo-isometric} embeddings (Definition~\ref{Def:PseudoIso}) and show that an embedding is pseudo-isometric if, and only if,
\begin{itemize}
	\item[\((S)\)] pairs of points of \(\Oct\) are at mutual distance \(2\) if, and only if, they are symplectic in \(\MS\).
\end{itemize}
 
Next (See Section~\ref{Sec:StronglyPolarisedEmbedding}, Theorem~\ref{Thm:StronglyPolarisedEmbedding}), we show that a pseudo-isometric embedding is \emph{strongly polarised} (Definition~\ref{Def:StronglyPolarised}) if, and only if, it is polarised (Definition~\ref{Def:Polarised}), that is, 
\begin{itemize}
	\item[\((P)\)] through every point \(x\) of \(\Oct\) there exists a symp \(x^\polarity\) of \(\MS\) such that all points collinear to \(x\) in \(\Oct\) are contained in \(x^\polarity\),
\end{itemize}

Then (See Section~\ref{Sec:Dual+OvoidalEmbedding}, Theorem~\ref{Thm:Dual+OvoidalEmbedding}), we show that a pseudo-isometric and strongly polarised embedding is \emph{\dual} (Definition~\ref{Def:Dual}) and \emph{ovoidal} (Definition~\ref{Def:Ovoidal}) if, and only if, it is \emph{dense} (Definition~\ref{Def:Dense}), that is, 
\begin{itemize}
	\item[\((D)\)] every symp of \(\MS\) contains at least one point of \(\Oct\).
\end{itemize}

Finally (Section~\ref{Sec:ReeTits}), we make the connection to the Ree-Tits Octagons and use the following consequence of Theorem 4.33 in \cite{Rijpert2025} as a lemma:
\begin{lemma}\label{Cor:KeyCor}
	Let \(\Oct\) be a generalised octagon fully embedded in a metasymplectic space \(\MS\).
	Then, \(\Oct\) is a Ree-Tits octagon naturally associated to a polarity of \(\MS\) if the embedding of \(\Oct\) in \(\MS\) is pseudo-isometric, strongly polarised, \dual\ and ovoidal.
\end{lemma}
We prove this lemma explicitly in Section~\ref{Sec:ReeTits}.

\section{Background}\label{Sec:Background}
Since we are only interested in generalised octagons and metasymplectic spaces, we do not intend to provide a very general introduction. Hence, one should be aware that some definitions here are more restrictive than in other publications; we specifically work our way to our main protagonists in an economical way.  
\begin{definition}[\bf Geometry]
	A \emph{geometry (of rank \(n\))}  is an \((n+1)\)-tuple \((\V_1,\V_2,\dots,\V_n,\C)\) where \(\V_1\), \(\V_2\), \dots, \(\V_n\) are disjoint non-empty sets and \(\C\) is a subset of \(\V_1\times\V_2\times\dots\times\V_n\).
\end{definition}

Let \(\Gamma = (\V_1,\V_2,\dots,\V_n,\C)\) be a geometry.
We call elements of \(V_1\) \emph{points}, elements of \(V_2\) \emph{lines} and elements of \(\C\) \emph{chambers}. 
In general, all elements of \(\V_1\cup\dots\cup \V_n\) will be referred to as \emph{elements of the geometry}.

Let \(x\) be an element of \(\V_i\) and \(y\) be an element of \(\V_j\) with \(i\neq j\) and \(i,j\in \{1,2,\dots,n\}\). We say that \(x\) and \(y\) are \emph{incident} and write \(x\I^\Gamma y\) if they are contained in a common chamber. 
Two points \(x\) and \(y\) of a geometry \(\Gamma\) are called collinear, denoted by \(x\col y\), if there exists a line incident with both.
The \emph{incidence graph} of a geometry is the graph with as vertices the elements of the geometry and edges when incident. 

The\emph{ collinearity graph} of a geometry is the graph with as vertices the points of the geometry, adjacent when collinear. 
We call a geometry \emph{connected} if its incidence graph is connected.	 Note that in a connected geometry, every element is contained in a chamber. 

\begin{definition}[\bf Generalised $n$-gon]
	A \emph{generalised \(n\)-gon} \(\Gamma\) of \emph{order $(s,t)$} is a rank~\(2\) geometry \((\P^\Gamma,\L^\Gamma,\C^\Gamma)\) which satisfies the following axioms:
	\begin{itemize}
		\item The incidence graph has diameter \(n\) and girth \(2n\).
		\item Every point is incident with $t+1\geq 3$ lines and every line is incident with $s+1\geq 3$ points.
	\end{itemize}
\end{definition}
Since in a rank 2 geometry, the incidence relation is essentially the same as the set of chambers, we sometimes write \((\P^\Gamma,\L^\Gamma,\I^\Gamma)\) instead of \((\P^\Gamma,\L^\Gamma,\C^\Gamma)\). 
\begin{definition}[\bf Dual]
	The \emph{dual}, denoted \(\Gamma^D\), of a geometry \(\Gamma = (\V_1,\V_2,\dots,\V_n,\C)\) is the geometry \((\V_n,\V_{n-1},\dots,\V_1,\C^D)\) where \[(v_1,v_2\dots,v_n)\in \C \iff (v_n,v_{n-1}\dots,v_1)\in \C^D.\]
\end{definition}
The following is obvious.
\begin{corollary}
	The dual of a generalised \(n\)-gon is itself a generalised \(n\)-gon.
\end{corollary}

\begin{corollary}\label{Cor:UniqueClosest}
	Let \(x\) be a point and \(L\) be a line of a generalised \(2n\)-gon. Then, there exists a unique point on \(L\) which is closest to \(x\) (using the distance in the incidence graph).
\end{corollary}
The following definitions are not standard, but rather convenient for us to later define metasymplectic spaces. Other names in the literature used for what we call shadows and residues are lower residues and upper residues, respectively.
\begin{definition}[\bf Shadow, Residue]
	Let \(\Gamma\) be a rank \(n\) geometry \((\V_1,\V_2,\dots,\V_n,\C)\) with \(n\) at least 2.
	
	The \emph{\(v\)-shadow-geometry} where \(v\) is an element of \(V_i\) with \(2\le i\le n\), is the geometry 
	\((\V_1',\V_2',\dots,\V_{i-1}',\C')\) with \(V_j' = \{w\in V_j| w\I^\Gamma v\}\) for all \( j \in \{1,\dots,i-1\}\) and \[(v_1,\dots,v_{i-1})\in \C' \iff \exists c\in \C \text{ with }  c = (v_1,\dots,v_{i-1},v, v_{i+1}\dots,v_n).\]
	
	The \emph{\(v\)-residue-geometry} where \(v\) is an element of \(V_i\) with \(1\le i\le n-1\) is the geometry 
	\((\V_{i+1}',\dots,\V_{n}',\C')\) with \(V_j' = \{w\in V_j| w\I^\Gamma v\}\) for all \( j \in \{i+1,\dots,n\}\) and \[(v_{i+1}\dots,v_{n})\in \C' \iff \exists c\in \C \text{ with }  c = (v_1, v_2,\dots,v_{i-1},v, v_{i+1}\dots,v_n).\]
	
\end{definition}

\begin{remark}
	It is easy to see from the definition that the \(v\)-residue-geometry in a geometry \(\Gamma\) is the dual geometry of the \(v\)-shadow-geometry in the dual geometry of \(\Gamma\).
\end{remark}



\begin{definition}[\bf Polar Space]
	A \emph{polar space} of rank $n$ is a geometry \(\Gamma = (\V_1,\V_2,\dots,\V_n,\C)\) which satisfies:
	\begin{enumerate}
		\item[(PS1)] Every line is incident with at least 3 points
		\item[(PS2)] No point is collinear with every other point
		\item[(PS3)] Let \(x\) be a point and \(L\) a line not incident with \(x\) then either there exists a unique point on \(L\) collinear to \(x\) or all points of \(L\) are collinear to \(x\).
	\end{enumerate}
\end{definition}

\begin{definition}[\bf Metasymplectic Space]\label{def:MSS}
	A \emph{metasymplectic space} \(\MS\) is a connected rank \(4\) geometry \((\PMS,\LMS,\PiMS,\SMS,\CMS)\) with \(|\SMS|\ge 2\), where we call the elements of \(\PMS\), \(\LMS\), \(\PiMS\) and \(\SMS\), respectively, \emph{points}, \emph{lines}, \emph{planes} and \emph{symplecta} (or \emph{symps} for short), and which satisfies:
	\begin{enumerate}
		\item[(MS1)] For each symp \(\Sigma\), the \(\Sigma\)-shadow-geometry is a polar space of rank \(3\).
		\item[(MS2)] For each point \(x\), the \(x\)-residue-geometry is a dual polar space of rank \(3\).
		\item[(MS3)] The intersection of two distinct symps is either empty, a point, a line or a plane. 
	\end{enumerate}
	
\end{definition}
One can check that the ``dual'' of (MS3) also holds, and so we have:
\begin{corollary}\label{Cor:dualMS}
	The dual of a metasymplectic space is itself a metasymplectic space.
\end{corollary}

We state some basic properties of metasymplectic spaces. Proofs can be found in \cite{Cohen}. In the following lemmas, we introduce notation and terminology. 

\begin{lemma}[\bf Point-Point Relation]\label{Lem:PP}
	Let \(x\) and \(y\) be two distinct points of a metasymplectic space, then one of the following is true:
	\begin{itemize}
		\item \(x\) and \(y\) are \emph{collinear}, \(x\col y\): There exists a unique line, \(xy\), which contains both \(x\) and \(y\).
		\item \(x\) and \(y\) are \emph{symplectic}, \(x\sympl y\): There exist a unique symp, \(\sympThrough{x}{y}\), which contains both \(x\) and \(y\).
		\item \(x\) and \(y\) are \emph{special}, \(x\special y\): There exists a unique point, \(\specPoint{x}{y}\), collinear to both \(x\) and \(y\).
		\item \(x\) and \(y\) are \emph{opposite}, \(x\opp y\): There are no points collinear to both \(x\) and \(y\).
	\end{itemize}
\end{lemma}


\begin{lemma}[\bf Symp-Symp Relation]\label{Lem:SympSymp}
	Let \(\Sigma_1\) and \(\Sigma_2\) be two distinct symps of a metasymplectic space, then one of the following is true:
	\begin{itemize}
		\item \(\Sigma_1\) and \(\Sigma_2\) are \emph{adjacent}: There exists a unique plane contained in both \(\Sigma_1\) and \(\Sigma_2\).
		\item \(\Sigma_1\) and \(\Sigma_2\) are \emph{symplectic}: There exist a unique point contained in both \(\Sigma_1\) and \(\Sigma_2\).
		\item \(\Sigma_1\) and \(\Sigma_2\) are \emph{special}: There exists a unique symp adjacent to both \(\Sigma_1\) and \(\Sigma_2\).
		\item \(\Sigma_1\) and \(\Sigma_2\) are \emph{opposite}: There are no symps adjacent to both \(\Sigma_1\) and \(\Sigma_2\).
	\end{itemize}
\end{lemma}

\begin{lemma}[\bf Point-Symp Relation]\label{Lem:PS}
	Let \(x\) be a point and \(\Sigma\) a symp which does not contain \(x\). Then one of the following is true:
	\begin{itemize}
		\item \(x\) is \emph{close} to \(\Sigma\): All points in \(\Sigma\) collinear to \(x\) form a line \(L\). Every point of \(\Sigma\) not on \(L\) is either symplectic to \(x\) and collinear to all points of \(L\) or special to \(x\) and is collinear to a unique point on \(L\).
		\item \(x\) is \emph{far} from \(\Sigma\): There is a unique symp through \(x\) which meets \(\Sigma\) in a point \(y\). Every point of \(\Sigma\) different from \(y\) is either special to \(x\) and collinear to \(y\) or opposite to \(x\) and symplectic to \(y\).
	\end{itemize}
\end{lemma}

\begin{corollary}\label{Cor:NoOppInAdj}
	Let \(\Sigma_1\) and \(\Sigma_2\) be two symps of a metasymplectic space.
	If there exist points \(x\in \Sigma_1\) and \(y\in\Sigma_2\) which are opposite, then \(\Sigma_1\) and \(\Sigma_2\) are not adjacent.
\end{corollary}

\begin{corollary}\label{Cor:NoColInOpp}
	Let \(\Sigma_1\) and \(\Sigma_2\) be two symps of a metasymplectic space.
	If there exist points \(x\in \Sigma_1\) and \(y\in\Sigma_2\) which are collinear, then \(\Sigma_1\) and \(\Sigma_2\) are not opposite.
\end{corollary}

\begin{corollary}[\bf Point-Line Relation]\label{Cor:PointLine}
	Let \(x\) be a point and let \(L\) be a line of a metasymplectic space. Then one of the following is true:
	\begin{itemize}
		\item \(x\) is on \(L\).
		\item One point on \(L\) is collinear to \(x\), the rest are symplectic to \(x\).
		\item One point on \(L\) is collinear to \(x\), the rest are special to \(x\).
		\item One point on \(L\) is symplectic to \(x\), the rest are special to \(x\). In this case, there exists a point \(y\) collinear to \(x\) and collinear to all points on \(L\).
		\item All points on \(L\) are special to \(x\). In this case, there exists a line \(M\) such that for any point \(y\) on \(L\), the point \(\specPoint{x}{y}\) is a different point on \(M\).
		\item One point on \(L\) is special to \(x\), the rest are opposite to \(x\).
	\end{itemize}
\end{corollary}

\begin{corollary}[\bf Plane-Symp Relation]\label{Cor:PlaneSymp}
	Let \(\Sigma\) be a symp and let \(\pi\) be a plane of a metasymplectic space. Then one of the following is true:
	\begin{itemize}
		\item \(\pi\) is in \(\Sigma\).
		\item One symp through \(\pi\) is adjacent to \(\Sigma\), the rest are symplectic to \(\Sigma\).
		\item One symp through \(\pi\) is adjacent to \(\Sigma\), the rest are special to \(\Sigma\).
		\item One symp through \(\pi\) is symplectic to \(\Sigma\), the rest are special to \(\Sigma\).
		\item All symps through \(\pi\) are special to \(\Sigma\).
		\item One symp through \(\pi\) is special to \(\Sigma\), the rest are opposite to \(\Sigma\).
	\end{itemize}
\end{corollary}

\begin{lemma}\label{Lem:SSOpp}
	Let \(x\) and \(y\) be two distinct points of a metasymplectic space \(\MS\).
	Then \(x\) and \(y\) are opposite if, and only if, there exist points \(x'\) and \(y'\) such that \(x\) and \(y'\) are special with \(\specPoint{x}{y'} = x'\), and \(y\) and \(x'\) are special with \(\specPoint{y}{x'} = y'\).
	Also, if $x\pperp z\perp y$, then $x$ and $y$ are not opposite.
\end{lemma}

\begin{corollary}\label{Cor:OpLinesOpPoints}
	Let \(\Sigma\) be a symp of a metasymplectic space \(\MS\) and let \(x\) and \(y\) be two points which are close to \(\Sigma\).
	Let \(L_x\) be the line through all points of \(\Sigma\) collinear to \(x\) and let \(L_y\) be the line through all points of \(\Sigma\) collinear to \(y\).
	The points \(x\) and \(y\) are opposite if and only if the lines \(L_x\) and \(L_y\) are opposite in \(\Sigma\).
\end{corollary}

\begin{lemma}\label{Lem:adjacentspecial}
Let \(\pi\) be a plane and let \(x\) and \(y\) be two points not in \(\pi\) such that \(x\) is collinear to a line \(L\) in \(\pi\)	and \(y\) is collinear to a line \(M\) in \(\pi\). Then, \(x\) and \(y\) are special if, and only if, \(L\) and \(M\) are different.
\end{lemma}

\begin{lemma}\label{Lem:specialplanes}
Let $\Sigma_1$ and $\Sigma_2$ be two special symps, and let $x_1\in\Sigma_1$ and  $x_2 \in \Sigma_2$ be symplectic, then at least one of $x_1,x_2$ is contained in the unique symp adjacent to both \(\Sigma_1\) and \(\Sigma_2\).
\end{lemma}

As the symps are polar spaces of rank \(3\) we get the following property: 
\begin{lemma}
	Let \(x\) be a point and \(\Sigma\) be a symp containing \(x\), then the lines and planes through \(x\) in \(\Sigma\) form a generalised quadrangle.
\end{lemma}
We refer to this generalised quadrangle as the \(\{x,\Sigma\}\)-GQ.

\textbf{Standing Hypothesis and Conventions.}
We now consider a generalised octagon \(\OctFull\) which is fully embedded in a metasymplectic space \(\MSFull\), that is, the points and lines of \(\Oct\) are points and lines of the metasymplectic space, the incidence in \(\Oct\) is determined by the incidence in \(\MS\) and for any line in \(\LO\subseteq \LMS\) all points of \(\PMS\) on it are in \(\PO\).
We will use the following convention: \emph{collinear}, \emph{symplectic}, \emph{special} and \emph{opposite} are always in the metasymplectic space while \emph{distance} refers to distance between the points in the collinearity graph of the generalised octagon. We write \(d(x,y)\) for this distance between two points \(x\) and \(y\) of \(\Oct\).

\section{Pseudo-isometric embeddings}\label{Sec:PseudoIsoEmbedding}
\begin{definition}[\bf Pseudo-isometric Embedding]\label{Def:PseudoIso}
	A generalised octagon \(\Oct\) is \emph{pseudo-isometrically embedded} in a metasymplectic space \(\MS\) if it is fully embedded and two points of the generalised octagon are at
	\begin{itemize}
		\itemsep0em
		\item distance \(1\) if, and only if, they are collinear,
		\item distance \(2\) if, and only if, they are symplectic,
		\item distance \(3\) if, and only if, they are special,
		\item distance \(4\) if, and only if, they are opposite.   
	\end{itemize}
\end{definition}

\subparagraph{}
The goal of this section is to show the following:

\begin{theorem}\label{Thm:PseudoIsoEmbedding}
	Let \(\Oct\) be a generalised octagon fully embedded in a metasymplectic space \(\MS\).
	Then the embedding of \(\Oct\) in \(\MS\) is pseudo-isometric if, and only if,
	\begin{itemize}
		\item[\((S)\)] two points of \(\Oct\) are at distance \(2\) if, and only if, they are symplectic.
	\end{itemize}
\end{theorem}

\subparagraph{}
By definition, any pseudo-isometric embedding satisfies Hypothesis~\((S)\).
We assume for the rest of this section that \(\Oct\) is a generalised octagon fully embedded in a metasymplectic space \(\MS\) and that this embedding satisfies Hypothesis~\((S)\).
We show that the embedding is pseudo-isometric.

\begin{lemma}\label{Lem:LinesNotPlanar}
	Let \(x\) be a point of \(\Oct\). Then, no two lines of \(\Oct\) through \(x\) are contained in a plane.
\end{lemma}
\begin{proof}
	Let \(L\) and \(M\) be two lines of \(\Oct\) through \(x\) and let \(y\) and \(z\) be points contained in \(L\setminus\{x\}\) and \(M\setminus\{x\}\), respectively. By Hypothesis~\((S)\), the points \(y\) and \(z\) are symplectic. Hence, the lines \(L\) and \(M\) can not be contained in a plane.
\end{proof}

\begin{lemma}\label{Lem:d3Special}
	Every two points of \(\Oct\) at distance \(3\) from each other are special.
\end{lemma}
\begin{proof}
	Let \(x\) and \(y\) be two points of \(\Oct\) at distance 3 from each other.
	Let \(x'\) and \(y'\) be the points of \(\Oct\) such that \[d(x,x') = d(x',y') = d(y',y) = 1.\]
	By Hypothesis~\((S)\), the points \(x\) and \(y'\) are symplectic. 
	The point \(y\) is collinear to \(y'\) and hence close to \(\sympThrough{x}{y'}\) (it is not contained in it by Hypothesis $(S)$). Let \(L\) be the line through all the points of \(\sympThrough{x}{y'}\) which are collinear to \(y\).
	Since \(x\) is not collinear to \(y'\) on \(L\), the point \(x\) is collinear to a unique point on \(L\) and special to \(y\) (Lemma~\ref{Lem:PS}).
\end{proof}

\begin{lemma}\label{Lem:dist4notcoll}
	No pair of points of \(\Oct\) at distance $4$ from each other are collinear.
\end{lemma}
\begin{proof}
	Assume for a contradiction that \(x\) and \(y\) are two points of \(\Oct\) at distance 4 from each other which are collinear.
	Let \(x'\) and \(y'\) be the unique points of \(\Oct\) such that \(d(x,x')=d(y,y')=1\) and \(d(x,y')=d(y,x')=3\).
	By Lemma~\ref{Lem:d3Special}, the points \(x\) and \(y'\) are special and similarly \(y\) and \(x'\) are special. We now see that \(\specPoint{x}{y'} = y\) and \(\specPoint{x}{y'} = x\). By Lemma~\ref{Lem:SSOpp} this means that \(x'\) and \(y'\) are opposite but this contradicts \(d(x',y') = 2\).
\end{proof}

Before we can prove that each pair of points at mutual distance 4 is opposite, we first need to show that at least one pair of such points is opposite.
We consider an octagon of \(\Oct\), that is, eight points \(x_0, x_1, \dots, x_7\) of \(\Oct\) such that, \(d(x_0,x_{1})=d(x_1,x_2)=\dots=d(x_6,x_7)=d(x_7,x_0)=1\), and \(d(x_0,x_4) = 4 = d(x_2,x_6)\).
For ease of notation, we define \(x_{i}=x_{i-8}\) for \(8\le i \le 15\). 
We see that: \[d(x_i,x_j)=j-i \text{ for }i\in \{0,1,\dots,7\}, i\le j \le i+4.\]

\begin{lemma}\label{Lem:S17S35NotEq}
	\(\sympThrough{x_1}{x_{7}} \) and \(\sympThrough{x_{3}}{x_{5}}\) are different symps.
\end{lemma}
\begin{proof}
	Follows from Lemma~\ref{Lem:d3Special}.
\end{proof}

\begin{lemma}\label{Lem:S17S35NotAdj}
	$\sympThrough{x_1}{x_{7}}$ and $\sympThrough{x_{3}}{x_{5}}$ are not adjacent.
\end{lemma}
\begin{proof}
	Suppose that these symps are adjacent and let \(\pi\) be the unique plane contained in both.
	Using Lemma~\ref{Lem:d3Special}, it is easy to see that \(\pi\) contains none of the points on $x_7x_0, x_0x_1, x_3x_4 $ or $ x_4x_5$. Hence there is a unique plane $\pi'$ containing $x_0x_7$ and a point $y$ of $\pi$. There is at least one point \(y'\) on $x_3x_4$ collinear to $y$. 
	This implies, by Corollary~\ref{Cor:PointLine}, that there is at least one point $z\in x_0x_7$ symplectic to $y'$. Note that, because the embedding of \(\Oct\) is full, \(y'\) and \(z\) are points of \(\Oct\). Since \(d(y',z)>2\), this contradicts Hypothesis~\((S)\). 
\end{proof}

\begin{lemma}\label{Lem:S17S35NotSympl}
	If $x_0$ and $x_4$ are not opposite, then $\sympThrough{x_1}{x_{7}}$ and $\sympThrough{x_{3}}{x_{5}}$ are not symplectic.
\end{lemma}
\begin{proof}
	Suppose, for a contradiction, that $x_0$ and $x_4$ are not opposite and that $\sympThrough{x_1}{x_{7}}$ and $\sympThrough{x_{3}}{x_{5}}$  are symplectic. Let $y$ be their intersection in $\MS$. It is easy to see that $y$ is not a point of $\Oct$ (otherwise $y$ is at distance~$2$ from both $x_0$ and $x_4$, implying by Hypothesis $(S)$ that is is also at distance $2$ from at least one point of $x_0x_1\setminus\{x_0\}$ and at least one point of $x_0x_7\setminus\{x_0\}$, clearly a contradiction). Since $x_1$ is symplectic to $x_3$ and $x_5$ is symplectic to $ x_7$, all of the points $x_1,x_3,x_5,x_7$ are collinear to $y$. If none of $x_0$ and $x_4$ were collinear to $y$, then they would be opposite, and the lemma would follow. Hence we may assume $x_0\col y$. Hence $y,x_0,x_7$ are contained in a plane $\pi$. As in the previous proof, by Corollary~\ref{Cor:PointLine}, there is at least one point on \(x_0x_7\) symplectic to \(x_3\), but no point of $x_0x_7$ is at distance 2 from any point of  $x_3x_4$, a contradiction. 
\end{proof}

\begin{lemma}\label{Lem:S17S35Spec}
	If $\sympThrough{x_1}{x_{7}}$ and $\sympThrough{x_{3}}{x_{5}}$ are special then, for some \(0\le i\le 7\), the symps $\sympThrough{x_i}{x_{i+2}}$ and $\sympThrough{x_{i+2}}{x_{i+4}}$ are not adjacent.
\end{lemma}
\begin{proof}
	Assume for a contradiction that the symps \(\sympThrough{x_1}{x_3}\) and \(\sympThrough{x_5}{x_7}\) are adjacent to both \(\sympThrough{x_1}{x_7}\) and \(\sympThrough{x_3}{x_5}\), and that the latter are special. Then \(\sympThrough{x_1}{x_3}=\sympThrough{x_5}{x_7}\) implying, by  Lemma~\ref{Lem:dist4notcoll}, that $x_2 $ and $ x_6$ are symplectic, contradicting Hypothesis~\((S)\). 
\end{proof}

\begin{lemma}\label{Lem:S17S35Opp}
	If $\sympThrough{x_1}{x_{7}}$ and $\sympThrough{x_{3}}{x_{5}}$ are opposite then the symps $\sympThrough{x_1}{x_{3}}$ and $\sympThrough{x_{3}}{x_{5}}$ are not adjacent.
\end{lemma}
\begin{proof}
	Assume $\sympThrough{x_1}{x_{7}}$ and $\sympThrough{x_{3}}{x_{5}}$ are opposite. Then, \(x_1\) is far from \(\sympThrough{x_{3}}{x_{5}}\) and hence, $\sympThrough{x_{1}}{x_{3}}$ is symplectic to $\sympThrough{x_{3}}{x_{5}}$.
\end{proof}

\begin{lemma}\label{Lem:NotAdjOpp}
	If, for some \(0\le i\le 7\), the symps $\sympThrough{x_i}{x_{i+2}}$ and $\sympThrough{x_{i+2}}{x_{i+4}}$ are not adjacent, then two of the points \(x_0, x_1, \dots, x_7\) at distance $4$ from each other 
	are opposite.
\end{lemma}
\begin{proof}
	Assume without loss of generality that $\sympThrough{x_0}{x_{2}}$ and $\sympThrough{x_{2}}{x_{4}}$ are not adjacent. 
	By lemma~\ref{Lem:d3Special}, \(x_0\) is special to \(x_3\in \sympThrough{x_{2}}{x_{4}}\). Hence, \(\sympThrough{x_0}{x_{2}} \neq \sympThrough{x_{2}}{x_{4}}\). Since \(x_2\) is contained in both \(\sympThrough{x_0}{x_{2}}\) and \(\sympThrough{x_{2}}{x_{4}}\), these symps are symplectic. It follows from Hypothesis~\((S)\) and Lemma~\ref{Lem:PS} that \(x_0\) and \(x_4\) are opposite.
\end{proof}

\begin{lemma}\label{Lem:ExitsOpposite}
	There exist two points of \(\Oct\) at distance $4$ from each other which are opposite.
\end{lemma}
\begin{proof}
	Consider the symps $\sympThrough{x_1}{x_{7}}$ and $\sympThrough{x_{3}}{x_{5}}$. By Lemma~\ref{Lem:S17S35NotEq}, they are not equal.
	By Lemma~\ref{Lem:S17S35NotAdj}, they are not adjacent. By Lemma~\ref{Lem:S17S35NotSympl}, we may assume that they are not symplectic. Hence, these are either special or opposite. Either by Lemma~\ref{Lem:S17S35Spec} or by Lemma~\ref{Lem:S17S35Opp}, we find that for some \(i\), \(0\le i\le 7\), the symps $\sympThrough{x_i}{x_{i+2}}$ and $\sympThrough{x_{i+2}}{x_{i+4}}$ are not adjacent.
	Using Lemma~\ref{Lem:NotAdjOpp}, we find two opposite points of \(\Oct\). 
\end{proof}

We now state Lemma 4.1 from \cite{G2inF4} for generalised octagons instead of hexagons. 
\begin{lemma}\label{Lem:inductionOnEvenPoly} 
	Let \(\circ\) be a binary symmetric relation between points of \(\Oct\) at mutual distance \(4\). Suppose that \(\circ\) has the following properties:
		\begin{itemize}
			\item There exist points \(x_1\) and \(x_2\) with \(d(x_1,x_2) = 4\) such that \(x_1 \circ x_2\).
			\item For every three points \(y_1\), \(y_2\) and \(y_3\) of \(\Oct\) such that \(d(y_1,y_2)=4=d(y_3,y_1)\) and \(d(y_2,y_3) = 1\), if \(y_1 \circ y_2\), then \(y_1\circ y_3\).
		\end{itemize}
	Then \(z_1\circ z_2\) for each pair of points of \(\Oct\) at mutual distance $4$.
\end{lemma}
\begin{proof}
The proof of Lemma 4.1 of \cite{G2inF4} can be taken over, replacing the distances $3$ and $2$ between points by $4$ and $3$, respectively. 
\end{proof}

\begin{lemma}\label{L20}
	Every two points of \(\Oct\) at distance $4$ from each other are opposite.
\end{lemma}
\begin{proof}
	By Lemma~\ref{Lem:ExitsOpposite}, there exist points \(x_1\) and \(x_2\) with \(d(x_1,x_2) = 4\) such that \(x_1\) and \(x_2\) are opposite.
	Let \(y_1\), \(y_2\) and \(y_3\) be three points of \(\Oct\) such that \(d(y_1,y_2)=4=d(y_3,y_1)\), \(d(y_2,y_3) = 1\) and \(y_1\) and \(y_2\) are opposite. The line \(y_2y_3\) contains \(y_2\) opposite to \(y_1\) and a unique point \(y_4\) at distance 3 from \(y_1\). By Lemma~\ref{Lem:d3Special}, \(y_1\) and \(y_4\) are special. Hence, Corollary~\ref{Cor:PointLine} implies that \(y_1\) and \(y_3\neq y_4\) are opposite. 
	Lemma~\ref{Lem:inductionOnEvenPoly}, with the relation `\emph{is opposite to}' proves the statement.
\end{proof}

\begin{corollary}
	The embedding is pseudo-isometric.
\end{corollary}

We have proven everything needed for Theorem~\ref{Thm:PseudoIsoEmbedding}.

\begin{remark}
	There is no known embedding which is not pseudo-isometric. We were not yet able to either construct such an embedding or show that such embeddings do not exist. A first step would be proving the existence or nonexistence of full embeddings of generalised octagons in symps, that is, in rank~3 polar spaces.
\end{remark}
\section{Strongly polarised embeddings}\label{Sec:StronglyPolarisedEmbedding}

\begin{definition}[\bf Polarised Embedding]\label{Def:Polarised}
	Let \(\Oct\) be a generalised octagon fully embedded in a metasymplectic space \(\MS\).
	Then, the embedding of \(\Oct\) is \emph{polarised} if, 
	\begin{itemize}
		\item[\((P)\)] through every point \(x\) of \(\Oct\) there exists a symp \(x^\polarity\) of \(\MS\) such that all points collinear to \(x\) in \(\Oct\) are contained in \(x^\polarity\).
	\end{itemize}
\end{definition}

\begin{definition}[\bf Strongly Polarised Embedding]\label{Def:StronglyPolarised} 
	Let \(\Oct\) be a generalised octagon fully embedded in a metasymplectic space \(\MS\).
	Then, the embedding of \(\Oct\) is \emph{strongly polarised} if:
	\begin{itemize}
		\item through every point \(x\) of \(\Oct\) there exists a unique symp \(x^\polarity\) of \(\MS\) such that all points collinear to \(x\) in \(\Oct\) are contained in \(x^\polarity\),
		\item two points \(x\) and \(y\) of \(\Oct\) are collinear, symplectic, special or opposite, respectively, if, and only if, the symps \(x^\polarity\) and \(y^\polarity\) are adjacent, symplectic, special or opposite, respectively.
	\end{itemize}
\end{definition}

\subparagraph{}
The goal of this section is to show the following:

\begin{theorem}\label{Thm:StronglyPolarisedEmbedding}
	Let \(\Oct\) be a generalised octagon fully embedded in a metasymplectic space \(\MS\) such that the embedding is pseudo-isometric.
	Then, the embedding of \(\Oct\) in \(\MS\) is strongly polarised if, and only if, it is polarised.
\end{theorem}

\subparagraph{}
Every embedding which is strongly polarised is by definition polarised.
We assume for the rest of this section that \(\Oct\) is a generalised octagon fully embedded in a metasymplectic space \(\MS\) such that the embedding is pseudo-isometric and polarised. 

\begin{lemma}\label{Lem:ExistsSymp}
	For every point \(x\) of \(\Oct\) there exists a unique symp of \(\MS\) which contains all points at distance 1 from \(x\).
\end{lemma}
\begin{proof}
	Let \(y\) and \(z\) be two points at distance 1 from \(x\). Since the embedding is pseudo-isometric, \(y\) and \(z\) are symplectic and are therefore contained in a unique symp. By Hypothesis~\((P)\), this symp \(\sympThrough{y}{z}\) contains all points at distance 1 from \(x\). 
\end{proof}

Let \(\polarity\) be the map from the points of \(\Oct\) to the symps of \(\MS\) which maps each point \(x\) to the unique symp containing all points at distance 1 from \(x\).

\begin{lemma}
	Let \(x\) and \(y\) be two points of \(\Oct\) which are collinear. Then, \(x^\polarity\) and \(y^\polarity\) are adjacent.
\end{lemma}
\begin{proof}
	Since the embedding is pseudo-isometric, \(x\) and \(y\) are at distance 1 from each other.
	Let \(x'\) and \(y'\) two other points of \(\Oct\) at distance 1 from respectively \(x\) and \(y\) and at distance 3 from each other. Because the embedding is pseudo-isometric, \(x'\in x^\polarity\) is special to \(y'\in y^\polarity\) and therefore \(x^\polarity\neq y^\polarity\).
	The statement follows from \(xy\) being contained in both \(x^\polarity\) and \(y^\polarity\).
\end{proof}

\begin{lemma}\label{Lem:SymplToSympl}
	Let \(x\) and \(y\) be two points of \(\Oct\) which are symplectic. Then \(x^\polarity\) and \(y^\polarity\) are also symplectic
\end{lemma}	
\begin{proof}
	Since the embedding is pseudo-isometric, \(x\) and \(y\) are at distance 2 from each other. Let \(z\) be the unique point of \(\Oct\) at distance 1 from both \(x\) and \(y\). The point \(z\) is contained in both \(x^\polarity\) and \(y^\polarity\) and hence these symps are either adjacent or symplectic. Let \(x'\) and \(y'\) be two points of \(\Oct\) at distance 1 from \(x\) and \(y\), respectively, and at distance~4 from each other. Because the embedding is pseudo-isometric, the points \(x'\) and \(y'\) are opposite.
	The symp \(x^\polarity\) contains \(x'\) and the symp \(y^\polarity\) contains \(y'\). Hence these symps can not be adjacent (Corollary~\ref{Cor:NoOppInAdj}) and are therefore symplectic.	
\end{proof}

\begin{lemma}\label{Lem:SpecToSpec}
	Let \(x\) and \(y\) be two points of \(\Oct\) which are special. Then \(x^\polarity\) and \(y^\polarity\) are also special.
\end{lemma}
\begin{proof}
	Since the embedding is pseudo-isometric, \(x\) and \(y\) are at distance 3 from each other. Let \(x'\) and \(y'\) be the unique points of \(\Oct\) such that \(x\col x' \col y' \col y\). Since \(x'\in x^\polarity\) and \(y'\in y^\polarity\) are collinear, the symps \(x^\polarity\) and \(y^\polarity\) can not be opposite (Corollary~\ref{Cor:NoColInOpp}).
	Let \(x''\) be a point of \(\Oct\) at distance 1 from \(x\) but at distance 4 from \(y\). Since the embedding is pseudo-isometric, \(x''\in x^\polarity\) and \(y\in y^\polarity\) are opposite and hence these symps are neither adjacent (Corollary~\ref{Cor:NoOppInAdj}). Similarly there exists a point \(y''\in y^\polarity\) collinear to \(y\) opposite \(x\).   
	Assume that there is a unique point \(z\) in both \(x^\polarity\) and \(y^\polarity\). Since \(x\) is opposite  \(y''\in y^\polarity\), the point \(x\) cannot be close to \(y^\polarity\) (Lemma~\ref{Lem:PS}). Therefore, \(z\) is the unique point in \(y^\polarity\) symplectic to \(x\). Since \(x\) and \(y\) are special, we see that \(y\) needs to be collinear to \(z\) (Lemma~\ref{Lem:PS}). This however contradicts the fact that \(x^\polarity\) contains \(x''\) opposite \(y\).
\end{proof}

\begin{lemma}\label{Lem:OppToOpp}
	Let \(x\) and \(y\) be two points of \(\Oct\) which are opposite. Then \(x^\polarity\) and \(y^\polarity\) are also opposite.
\end{lemma}

\begin{proof}
	Since $x^\polarity$ and $y^\polarity$ contain respective opposite points, we see that they are either symplectic, special or opposite. 
	Assume first, for a contradiction, that they are symplectic and let $z$ be their unique common point. Since every point of $\Oct$ collinear to $x$ is at distance~$4$ to at least one point of $\Oct$ at distance~$1$ from $y$, no point of $\Oct$ in $x^\polarity$ is collinear to $z$ (Lemma~\ref{Lem:PS}), a contradiction since every line in $x^\polarity$ contains at least one point collinear to $z$.
	
	Now assume, for a final contradiction, that $x^\polarity$ and $y^\polarity$ are special. 
	Let \(\Sigma\) be the unique symp adjacent to both \(x^\polarity\) and \(y^\polarity\).
	Set $\pi_x:=x^\polarity\cap\Sigma$ and $\pi_y=y^\polarity\cap \Sigma$. Let $L_i$, $i=1,2,3$, be three lines of $\Oct$ through $x$ and let $a_i\in L_i$ be the unique point at distance~$3$ from $y$. Let $b_i$ be the unique point of $\Oct$ collinear to $y$ and at distance $2$ from $a_i$. Then $a_i\sympl b_i$. Lemma~\ref{Lem:specialplanes} yields $a_i\in \pi_x$ or $b_i\in\pi_y$. So, without loss of generality, we have $a_1,a_2\in\pi_x$, contradicting the fact that $a_1$ and $a_2$ are symplectic because they are at distance~$2$ in $\Oct$.  
\end{proof}	

Combining the four previous lemmas, we obtain the following consequence.	

\begin{corollary}\label{Cor:dualiso}
	Two points \(x\) and \(y\) of \(\Oct\) are collinear, symplectic, special or opposite if, and only if, the symps \(x^\polarity\) and \(y^\polarity\) are adjacent, symplectic, special or opposite, respectively.
\end{corollary}

We have proven everything needed for Theorem~\ref{Thm:StronglyPolarisedEmbedding}.

\begin{remark}
	There is no known example of an embedding which is pseudo-isometric but not strongly polarised. Hence we can conjecture that every pseudo-isometric embedding is strongly polarised.
	Using Theorem~\ref{Thm:StronglyPolarisedEmbedding} it would suffice to show that any pseudo-isometric embedding is polarised, that is,  satisfies hypothesis \((P)\).
\end{remark}
\section{\Dual\ and Ovoidal Embeddings}\label{Sec:Dual+OvoidalEmbedding}

\begin{definition}[\bf \Dual\ Embedding]\label{Def:Dual}
	Let \(\Oct\) be a generalised octagon fully embedded in a metasymplectic space \(\MS\).
	Then, the embedding is called \emph{\dual} if there exists a map \(\polarity\) from the points and lines of \(\Oct\) to the symps and planes of \(\MS\) such that:
	\begin{itemize}
		\item If \(x\) is a point of \(\Oct\) then \(x^\polarity\) is a symp containing all points at distance~1 from \(x\).
		\item If \(L\) is a line of \(\Oct\) then \(L^\polarity\) is a plane and for every point \(x\) on \(L\), the symp \(x^\polarity\) contains \(L^\polarity\).
		\item The set of all the symps in the image of \(\polarity\) together with the set of all planes in the image of \(\polarity\) forms a generalised octagon \(\overline{\Oct}\) fully embedded in \(\overline{\MS}\).
	\end{itemize}
\end{definition}

\begin{remark}
	There are examples of strongly polarised embeddings that are not \dual:
	Let \(\Oct\) be a finite Ree-Tits octagon associated to the finite field $\mathbb{F}_q$, $q=2^{2e+1}$, $e\geq 0$, naturally embedded in a metasymplectic space \(\MS\) (see Section~\ref{Sec:ReeTits}). Then \(\MS\) is naturally associated to the Chevalley group of type $\mathsf{F_4}(q)$ and hence fully embedded in the  larger metasymplectic space \(\MS'\) naturally associated to the twisted Chevalley group $\mathsf{^2E_6}(q)$. Then, \(\Oct\) is strongly polarised but not \dual\ in \(\MS'\) since lines of $\overline{\MS'}$ have $q^2+1$ points and those of $\Oct$ only $q+1$.  Note that \(\Oct\) is strongly polarised and \dual\ in \(\MS\), see Section~\ref{Sec:ReeTits}.
\end{remark}
The last notion we need to apply Lemma~\ref{Cor:KeyCor} is that of ovoidal embedding. 
\begin{definition}[\bf Ovoidal Embedding]\label{Def:Ovoidal}
	Let \(\Oct\) be a generalised octagon fully embedded in a metasymplectic space \(\MS\) such that the embedding is \dual.
	Then, the embedding of \(\Oct\) is \emph{ovoidal} if for any point \(x\) of \(\Oct\) the following holds:
	\begin{itemize}
		\item The lines of \(\Oct\) in the \(\{x,x^\polarity\}\)-GQ form an ovoid of that generalised quadrangle.
		\item The planes of \(\overline{\Oct}\) in the\(\{x,x^\polarity\}\)-GQ form an ovoid in the dual generalised quadrangle.
	\end{itemize}
\end{definition}
Finally, one of the main hypotheses of the present paper is the following.

\begin{definition}[\bf Dense Embedding]\label{Def:Dense}
Let \(\Oct\) be a generalised octagon fully embedded in a metasymplectic space \(\MS\). Then, the embedding of \(\Oct\) is called \emph{dense}, if 
\begin{itemize}
		\item[\((D)\)] every symp of \(\MS\) contains at least one point of \(\Oct\).
	\end{itemize}
\end{definition}
 
The goal of this section is to show the following:

\begin{theorem}\label{Thm:Dual+OvoidalEmbedding}
	Let \(\Oct\) be a generalised octagon fully embedded in a metasymplectic space \(\MS\) such that the embedding is pseudo-isometric and strongly polarised.
	Then, the embedding of \(\Oct\) in \(\MS\) is \dual\ and ovoidal if, and only if, it is dense.	
\end{theorem}

\subparagraph{}
We first assume that \(\Oct\) is a generalised octagon which is fully embedded in a metasymplectic space \(\MS\) such that the embedding is pseudo-isometric, strongly polarised, \dual\ and ovoidal. Let \(\polarity\) be the map from the points of \(\Oct\) to the symps of \(\MS\) which maps each point \(x\) to the unique symp containing all points at distance~1 from \(x\).
Since the embedding is pseudo-isometric, it is clear that no symp of \(\MS\) can contain more than the set of points at distance 1 from a point of \(\Oct\). We now show that every symp of \(\MS\) contains at least one point of \(\Oct\). 

\begin{lemma}\label{Lem:OvoidalPointInPlane}
	Let \(x^\polarity\) be a symp of \(\overline{\Oct}\). Then, every plane \(\pi\) of \(x^\polarity\) contains at least one point of \(\Oct\).
\end{lemma}
\begin{proof}
	It is easy to see that there exists a line \(L\) in \(\pi\) such that all points on \(L\) are collinear to \(x\). Since the embedding is ovoidal, \(\langle x, L\rangle\) contains a line of \(\Oct\). Hence, \(L\), and therefore also \(\pi\), contains at least one point of \(\Oct\).
\end{proof}

This result implies the following dual result.
\begin{corollary}\label{Cor:OvoidalSympThroughLine}
	Let \(x\) be a point of \(\Oct\). Then, every line \(L\) of \(\MS\) through \(x\) is contained in at least one symp of \(\overline{\Oct}\).
\end{corollary}

\begin{lemma}\label{Lem:ClosePoint}
	For every symp \(\Sigma\) of \(\MS\) there exists a point of \(\Oct\) which is close to \(\Sigma\).
\end{lemma}
\begin{proof}
	Let \(L\) be a line of \(\Oct\). Either at least one of the symps through \(L^\polarity\) is adjacent to \(\Sigma\) or at least one of these symps through \(L^\polarity\) is special to \(\Sigma\) (Corollary~\ref{Cor:PlaneSymp}). Observe that, since the embedding of $\bar{\Oct}$ is full in \(\bar{\MS}\),  all these symps are in \(\overline{\Oct}\).
	
	\paragraph{\textbf{Case $1$}: There exists a symp \(y^\polarity\) of \(\overline{\Oct}\) adjacent to \(\Sigma\)} Then, $y$ is close to $\Sigma$.
	
	\paragraph{\textbf{Case $2$}: There exists a symp \(y^\polarity\) of \(\overline{\Oct}\) special to \(\Sigma\)} 
	Let \(\Sigma'\) be the unique symp adjacent to both \(y^\polarity\) and \(\Sigma\).
	Since \(\Oct\) is ovoidal, the plane \(y^\polarity\cap \Sigma'\) contains a point \(z\) of \(\Oct\) (Lemma~\ref{Lem:OvoidalPointInPlane}). Clearly, $z$ is close to $\Sigma$.
\end{proof}

\begin{lemma}\label{Lem:CloseLine}
	For every symp \(\Sigma\) of \(\MS\) there exists a line \(L\) of \(\Oct\) and a point \(x\) in \(\Sigma\) such that \(x\) is collinear to every point of \(L\).
\end{lemma}
\begin{proof}
	By Lemma~\ref{Lem:ClosePoint} there exists a point \(y\) of \(\Oct\) close to \(\Sigma\). 
	
	\paragraph{\textbf{Case \(1\)}: \(y^\polarity\) is adjacent to \(\Sigma\)}
	In this case for every line \(L\) of $y^\polarity$ through \(y\) there exists a point \(x\) in \(y^\polarity\cap\Sigma\) which is collinear to all points on \(L\) (and we can choose $L\in\LO$).
	
	\paragraph{\textbf{Case \(2\)}: \(y^\polarity\) is symplectic to \(\Sigma\)}
	Let \(x\) be the unique point in both \(y^\polarity\) and \(\Sigma\) and note that $x\col y$. Since the embedding is ovoidal, any plane through \(xy\) in \(y^\polarity\) contains a line \(L\) of \(\Oct\).
	
	\paragraph{\textbf{Case \(3\)}: \(y^\polarity\) is special to \(\Sigma\)}
	Let \(\Sigma'\) be the unique symp adjacent to both \(y^\polarity\) and \(\Sigma\). The point \(y\) is contained in \(y^\polarity\cap\Sigma' =: \pi\). Since the embedding is ovoidal, this plane \(\pi\) in \(y^\polarity\) through \(y\) contains a line of \(\Oct\).
\end{proof}

\begin{lemma}\label{Lem:nondisjointsymps}
	For every symp \(\Sigma\) of \(\MS\) there exists a symp $\Sigma'$ of \(\overline{\Oct}\) which is not disjoint from \(\Sigma\). Moreover, if no symp of \(\overline{\Oct}\) is adjacent to $\Sigma$, then we can choose $\Sigma'$ in such a way that it intersects $\Sigma$ in a unique point $x$, which is collinear to all points of a line $L\subseteq \Sigma'\setminus\Sigma$, with $L\in\LO$.
\end{lemma}
\begin{proof}
	By Lemma~\ref{Lem:CloseLine}, there exists a point \(x\) in \(\Sigma\) and a line \(L\) of \(\Oct\) such that \(x\) is collinear to every point of \(L\). If $L$ intersects $\Sigma$, then we clearly can find a symp of \(\overline{\Oct}\) adjacent to $\Sigma$. Hence we may assume that $L\cap\Sigma=\varnothing$. Any point \(y\) of \(L^\polarity\) not on \(L\) is clearly either collinear or symplectic to \(x\). Since all symps containing \(L^\polarity\) are symps of \(\overline{\Oct}\), any symp through $L$ and $x$ (if $x\col y$), or the symp \(\sympThrough{x}{y}\) (if $x\sympl y$) is a symp of \(\overline{\Oct}\).
\end{proof}

\begin{lemma}
	The following statements are true:
	\begin{itemize}
		\itemsep0em
		\item[\((i)\)] Every symp of \(\MS\) contains at least one point of \(\Oct\).
		\item[\((ii)\)] Every point of \(\MS\) is collinear to a point of \(\Oct\)
		\item[\((iii)\)] Every point of \(\MS\) is contained in a symp of \(\overline{\Oct}\).
		\item[\((iv)\)] Every symp of \(\MS\) is adjacent to a symp of \(\overline{\Oct}\).
	\end{itemize}
\end{lemma}
\begin{proof}
	Let \(\Sigma\) be a symp in \(\MS\) not contained in \(\overline{\Oct}\). By Lemma~\ref{Lem:nondisjointsymps}, either \((iv)\) holds or there exists a symp \(\Sigma'\) of \(\overline{\Oct}\) intersecting \(\Sigma\) in a point $x$, which is collinear to each point of a line \(L\subseteq\Sigma'\setminus\Sigma\), with $L\in\LO$. Let $p'\in\PO$ be such that ${p'}^\polarity=\Sigma'$. Then  $p'\in L$. Our assumption that \(\overline{\Oct}\) is ovoidal implies that there exists a plane \(\pi\) of \(\Oct\) with \(x,p'\in \pi\) and such that all symps through \(\pi\) belong to \(\overline{\Oct}\). Now \((iv)\) follows. 
	
	Now let, with $\Sigma$ as in the previous paragraph, \(\Sigma''\) be a symp of \(\overline{\Oct}\) adjacent to \(\Sigma\), and let $p''\in\PO$ be such that ${p''}^\polarity=\Sigma''$. If \(p''\in\Sigma\), then \((i)\) holds. If not, then since the embedding is ovoidal, there is a line \(L\) of \(\Oct\) through \(p''\) in the plane spanned by \(p''\) and the points of the plane \(\Sigma\cap\Sigma''\) collinear tp $p''$. Then \(L\) intersects \(\Sigma\) nontrivially and \((i)\) holds after all. 
	
	Dually, \((ii)\) and \((iii)\) hold true.
\end{proof}

For the rest of this section, we assume that \(\Oct\) is a generalised octagon fully embedded in a metasymplectic space \(\MS\) such that the embedding is pseudo-isometric, strongly polarised and dense.

\begin{lemma}	
	The embedding of \(\Oct\) is \dual.
\end{lemma}
\begin{proof}
	We first show that for every line \(L\) of \(\Oct\) there exists a unique plane \(\pi\) such that for every point \(x\) on \(L\), the symp \(x^\polarity\) contains \(\pi\), and every symp of $\MS$ through $\pi$ arises in this way.
	
	Let $x$ and $y$ be two points on \(L\). Then, $x^\polarity$ and $y^\polarity$ intersect in a plane $\pi$. 
	
	Let $\Sigma$ be an arbitrary symp containing $\pi$, and assume $\Sigma\notin\{x^\polarity,y^\polarity\}$. We claim that there exists a point $z\in xy$ such that $z^\polarity=\Sigma$. Indeed, let $\alpha$ be a plane of $\Sigma$ intersecting $\pi$ in a unique point $p$ not on \(xy\).
	If $\alpha$ contains some point $a$ of $\Oct$, then it is collinear to a unique point $z$ of $xy$ and hence at distance~$1$ from $z$ (pseudo-isometric). Consequently, $\sympThrough{x}{a}=\Sigma=z^\polarity$. Let $\Sigma'$ be a symp through $\alpha$ distinct from $\Sigma$. By Hypothesis~\((D)\) and the previous argument, we may assume that there is a point $a\in \Sigma'\setminus \alpha$ that belongs to $\Oct$.
	If $a$ were not collinear to $p$, then all points of $xy$ would be special to $a$ (Lemma~\ref{Lem:adjacentspecial}), contradicting the embedding being pseudo-isometric and the fact that in a generalised octagon no line exists all of whose points are at distance $3$ from a given point. 
	Note that for any point on \(xy\), there is a line in \(\alpha\) of points collinear to it ($\Sigma$ is a rank~3 polar space). This means that no point of \(xy\) is collinear to any point of \(\Sigma'\) not in \(\alpha\), in particular no point of $xy$ is collinear to \(a\).
	It follows that there exists a unique point $z\in xy$ that is symplectic to $a$. 
	If $w$ is the unique point of $\Oct$ at distance~$1$ from both $z$ and $a$, then $w^\polarity = \sympThrough{a}{z}\supseteq pz$.  Hence the symp $w^\polarity$ is adjacent with both $x^\polarity$ and $y^\polarity$ (as they all have the line $pz$ in common), which implies, since the embedding is pseudo-isometric and strongly polarised (Corollary~\ref{Cor:dualiso}), that $w$ is at distance~$1$ from both $x$ and $y$, hence contained in $xy$, a contradiction as $a$ is not collinear to any point of $xy$. The claim is proved.
	
	Now assume for a contradiction that there exists some point $z$ on \(L\) such that $z^\polarity$ intersects $x^\polarity$ in a plane $\pi'$ distinct from $\pi$. 
	
	Then, by the foregoing, all symps through $\beta$ belong to the image of $\polarity$. Since, by (MS2) of Definition~\ref{def:MSS}, the planes and symps through the line $xy$ form a projective plane, this means that the image of $\polarity$ in that projective plane (where we view the points of the latter as the symps through $xy$) is a subspace containing a triangle. We conclude that every symp through $xy$ coincides with $t^\polarity$, for some $t\in xy$. Now let $q$ be a point of \(\Oct\) at distance~$1$ from $x$ and at distance~$2$ from $y$. Then $q^\polarity\cap x^\polarity$ is a plane $\gamma$. In $x^\polarity$, there is a plane $\gamma_y$ through $y$ intersecting $\gamma$ in a line. Every symp through $\gamma_y$ distinct from $x^\polarity$ coincides with $t^\polarity$, for some $t\in xy\setminus\{x\}$.  But, for such a $t$, $t^\polarity$ is adjacent to $q^\polarity$, hence $t\perp q$ in $\Oct$, a contradiction.  
	
	So, we have shown that $xy$ corresponds to a unique plane $\pi=:(xy)^\polarity$, and all symps through $(xy)^\polarity$ are in the image of $\polarity$. Together with the embedding being pseudo-isometric and strongly polarised, this means that $\overline{\Oct}$ is fully embedded generalised octagon in $\overline{\MS}$.
\end{proof}

We now show that the embedding of \(\Oct\) is ovoidal.

\begin{lemma}\label{Lem:OvoidalDirect}
	For any point \(x\) of \(\Oct\), the lines of \(\Oct\) in the \(\{x,x^\polarity\}\)-GQ form an ovoid.
\end{lemma}
\begin{proof}
Let \(\pi\) be a plane through \(x\) in \(x^\polarity\) and let \(L\) be a line in \(\pi\) which does not contain \(x\). Since the embedding is pseudo-isometric, showing that there is a unique line of \(\Oct\) in \(\pi\) through \(x\) is equivalent to showing that \(L\) contains exactly one point of \(\Oct\).
Since \(\Oct\) does not contain any triangle and the embedding is pseudo-isometric it is clear that there is at most one point of \(\Oct\) on \(L\). Now let \(\Sigma\) be a symp through \(L\) not containing $x$. By our assumption, \(\Sigma\) contains a point \(y\) of \(\Oct\). If \(y\) is on \(L\) we are done so we may assume this is not the case.
Since \(\Sigma\) is a polar space, \(y\) is collinear to either a unique point on \(L\) or all points on \(L\).

\paragraph{Assume first that \(y\) is collinear to a unique point on \(L\)} 
The intersection \(x^\polarity\cap\Sigma\) is plane containing \(L\) and hence \(y\) is not a point in \(x^\polarity\). The point \(x\) is close to \(\Sigma\) with \(L\) as the unique line containing points of \(\Sigma'\) collinear to \(x\) (Lemma~\ref{Lem:PS}). Hence, \(x\) and \(y\) are special. Since the embedding is pseudo-isometric this implies that \(x\) and \(y\) are at distance 3 in $\Oct$. All points of \(\Oct\) at distance 1 from \(x\) are contained in \(x^\polarity\). Hence there is a point \(z\) of \(\Oct\) in $x^\polarity$ at distance \(4\) from \(y\). As the embedding is pseudo-isometric, \(y\) and \(z\) are opposite. However, $y$ is close to $x^\polarity$ as it is collinear to a point of $L\subseteq x^\polarity$, contradicting Lemma~\ref{Lem:PS}. Hence this situation does not occur.

\paragraph{Assume now that \(y\) is collinear to all points of \(L\)}
This implies that \(x\) and \(y\) are symplectic. Since the embedding is pseudo-isometric, \(x\) and \(y\) are at distance \(2\) from each other. All points of \(\Oct\) at distance 1 from \(x\) are contained in \(x^\polarity\) and hence \(x^\polarity\) also contains a point \(x'\) of \(\Oct\) at distance \(3\) from \(y\). Again, the embedding is pseudo-isometric which now implies that \(y\) and \(x'\) are special. Therefore, since \(x^\polarity\) contains a point special to \(y\), the point \(y\) can not be contained in \(x^\polarity\). 
The point \(y\) is collinear to \(L\) in \(x^\polarity\) and hence close to \(x^\polarity\) (Lemma~\ref{Lem:PS}).
Let \(z\) be the unique point of \(\Oct\) at distance \(1\) from both \(x\) and \(y\). All points of \(\Oct\) at distance \(1\) from \(x\) are contained in \(x^\polarity\) so in particular also \(z\). The points of \(x^\polarity\) collinear to \(y\), which now includes \(z\), are all on \(L\) (Lemma~\ref{Lem:PS}).
We have found the point of \(\Oct\) on \(L\) which we were looking for.
\end{proof}

\begin{lemma}
	For any point \(x\) of \(\Oct\), the planes of \(\Oct\) in the \(\{x,x^\polarity\}\)-GQ form a dual ovoid.
\end{lemma}
\begin{proof}
	Let \(L\) be a line of \(\MS\) in \(x^\polarity\) through $x$. We need to show that there exists a line \(M\) of \(\Oct\) through $x$ such that \(M^\polarity\) contains \(L\).
	We may assume that \(L\) is not a line of \(\Oct\), otherwise this statement is trivial. 
	Let \(z\) be a point on \(L\) different from \(x\) and let $\Sigma$ be a symp through $z$ symplectic to $x^\polarity$. Hypothesis~\((D)\) yields a point $y\in\PO\cap\Sigma$. Noting that $y\neq z$, and hence $d(y,x)\geq 2$, there are three cases to consider.
	
	\textbf{Case 1.} \emph{The point $y$ is symplectic to $z$.} \\
In this case $d(y,x)=3$. Let $y_0\in\PO$ be at distance~1 from $x$ and distance~2 from $y$. Then $y_0\in x^\polarity$ and is symplectic to $y$. This yields $y_0=z$, a contradiction. 

	\textbf{Case 2.} \emph{The point $y$ is collinear to $z$ and special to $x$.}\\
Here we also have $d(y,z)=3$. Each point of $\Oct$ at distance 1 from $x$ and distance $4$ from $y$ lies in $x^\polarity$ and is opposite $y$, a contradiction, as this would imply that $y$ is symplectic to $z$. Hence this case cannot occur. 

	\textbf{Case 3.} \emph{The point $y$ is collinear to $z$ and symplectic to $x$.}\\
In this case $d(y,x)=2$. The symp $\sympThrough{y}{x}$ intersects $x^\polarity$ in a plane c$\pi$ containing $L$ and coincides with $u^\polarity$ for the unique point $u$ of $\Oct$ at distance~$1$ from both $x$ and $y$. Hence $L\subseteq\pi=(xu)^\polarity$, which completes the proof of the lemma. 
\end{proof}

We have now proven everything needed for Theorem~\ref{Thm:Dual+OvoidalEmbedding}.

	

\begin{remark}
	There is no known example of a \dual\ embedding which is not ovoidal. Moreover, if we let \(\Oct\) be a fully embedded generalised octagon in a metasymplectic space \(\MS\) such that the embedding is \dual\ and the generalised octagon has an order \((q,q^2)\) where \(q\) is finite, then a counting argument\footnote{This will be included in full in the future PhD thesis of Sebastian Petit.} shows that hypothesis \((D)\) is satisfied. Hence, by Theorem~\ref{Thm:Dual+OvoidalEmbedding}, the embedding is also ovoidal. Therefore, we conjecture that every \dual\ embedding is ovoidal.
\end{remark}

\section{The Ree-Tits octagons}\label{Sec:ReeTits}
We first recall how Ree-Tits octagons are defined and show that the corresponding embeddings satisfy all properties of our characterisations. In the following, a polarity of a metasymplectic space is an isomorphism of order $2$ to its dual.  An absolute element of a polarity is an element incident with its image.
\begin{theorem}[Tits (unpublished), proof in \cite{VanMaldeghem1998} (Theorem~2.5.2)]\label{RTOct}
	Let \(\MS\) be a metasymplectic space over some field \(\K\), i.e., the planes of \(\MS\) are planes over \(\K\). Suppose \(\MS\) admits a polarity \(\polarity\). Then, the geometry \(\Oct\) whose points and lines are the absolute points and lines, respectively with the natural incidence relation, is a generalised octagon of order \((|\K|,|\K|^2)\).
\end{theorem}

\begin{definition}
	The generalised octagon $\Oct$ in the statement of Theorem~{\ref{RTOct}} is called a  \emph{Ree-Tits octagon}, and the corresponding embedding in $\MS$ is called the \emph{natural embedding of $\Oct$}. 
\end{definition}

\begin{corollary}\label{Cor:RTPolarised}
	Let \(\Oct\) be a Ree-Tits octagon naturally embedded in a metasymplectic space \(\MS\) with associated polarity $\polarity$.
	Then, for each point $x$, the symp $x^\polarity$ of $\MS$ contains all points of $\Oct$ collinear to $x$. 
\end{corollary}

\begin{proof}
If $x\in L\in\LO$, then applying $\polarity$ and taking into account that $L$ is absolute, yields $x\in L\subseteq L^\polarity\subseteq x^\polarity$.
\end{proof}

Applying the corresponding polarity to the natural embedding of a Ree-Tits octagon, we immediately obtain:
\begin{corollary}\label{Cor:RTDual}
	Let \(\Oct\) be a Ree-Tits octagon embedded in a metasymplectic space \(\MS\).
	Then, the embedding of \(\Oct\) is \dual.
\end{corollary}

\begin{lemma}[See \cite{VanMaldeghem1998}, Theorem~2.5.2, lemma~2]\label{Lem:RTovoidal}
	Let \(\Oct\) be a Ree-Tits octagon naturally embedded in a metasymplectic space \(\MS\).
	Then, the embedding of \(\Oct\) is ovoidal.
\end{lemma}
The previous lemma already implies that two points at mutual distance $2$  in a Ree-Tits octagon $\Oct$ are symplectic in the metasymplectic space in which $\Oct$ is naturally embedded. The converse is shown in \cite{VanMaldeghem1998}, Theorem~2.5.2, lemma~4. Hence we can state without further proof:
\begin{lemma}\label{Lem:RTd2}
	Let \(\Oct\) be a Ree-Tits octagon naturally embedded in a metasymplectic space~\(\MS\).
	Then, 
	\begin{itemize}
		\item[\((S)\)] two points of \(\Oct\) are at distance \(2\) if, and only if, they are symplectic.
	\end{itemize}
\end{lemma}

\begin{corollary}\label{Lem:RTpseudoiso}
	Let \(\Oct\) be a Ree-Tits octagon embedded in a metasymplectic space \(\MS\).
	Then, the embedding of \(\Oct\) is pseudo-isometric.
\end{corollary}
\begin{proof}
	Follows from Lemma~\ref{Lem:RTd2} and Theorem~\ref{Thm:PseudoIsoEmbedding}.
\end{proof}

We now first show the converse of the previous lemmas by proving Lemma~\ref{Cor:KeyCor}.
\begin{proof}[Proof of Lemma \emph{\ref{Cor:KeyCor}}]
Let \(\Oct\) be a generalised octagon fully embedded in a metasymplectic space \(\MS\). Suppose 
the embedding of \(\Oct\) in \(\MS\) is pseudo-isometric, strongly polarised, \dual\ and ovoidal. We define a subcomplex $\Xi$ of the building $\Delta$ corresponding to $\MS$ by collecting all pairs $\{x,x^\polarity\}$, with $x$ a point of $\Oct$ and $x^\polarity$ the symp containing all points of $\Oct$ collinear to $x$, and all pairs $\{L,L^\polarity\}$, where $L$ is a line of $\Oct$ and $L^\polarity$ is the plane $x^\polarity\cap y^\polarity$, for $x,y\in L$. Then $\Xi$ is the subcomplex induced on these simplices.  According to Theorem~4.33 of \cite{Rijpert2025}, we have to show that this is an \emph{ovoidal} subcomplex (in the sense of \cite{Rijpert2025}) and that, if $\{v_1,v_4\}$ and $\{v_1',v_4'\}$ are point-symp pairs belonging to $\Sigma$, with $v_1,v_1'$ opposite points, then $v_4$ and $v_4'$ are opposite symps. The latter condition is automatic by the definition of strongly polarised. For the former condition, we recall what is meant with ``ovoidal'' in \cite{Rijpert2025}. 
\begin{enumerate}
\item[(1)] Given a pair $\{L,L^\polarity\}\in\Xi$, with $L$ a line of $\Oct$, the mapping $\polarity:\PO\to \SMS:x\mapsto x^\polarity$ is a bijection from the points on $L$ to the set of symps containing $L^\polarity$. 
\item[(2)] Given a pair $\{x,x^\polarity\}\in\Xi$, with $x\in \PO$, the set $\{L\in \LMS \mid x\in L\subseteq L^\polarity\subseteq x^\polarity\}$ is an ovoid of the $(x,x^\polarity)$-GQ, and the the set $\{L^\polarity\in \PiMS \mid x\in L\subseteq L^\polarity\subseteq x^\polarity\}$ is an ovoid of  the dual of the $(x,x^\polarity)$-GQ,
\end{enumerate}

Condition~(1) is satisfied since \(\overline{\Oct}\) is full in \(\overline{\MS}\) by definition of the embedding being \dual. Condition~(2) is exactly the definition of ovoidal in the present paper and is hence satisfied. This shows Lemma~\ref{Cor:KeyCor}
\end{proof}

We can now combine our results to prove Theorem~\ref{Main}.
\begin{proof}[Proof of Theorem \emph{\ref{Main}}]
	Assume first that \(\Oct\) is a Ree-Tits octagon naturally associated to a polarity of \(\MS\).
	By Lemma~\ref{Lem:RTd2}, Condition $(S)$ is satisfied.
	By Corollary~\ref{Cor:RTPolarised}, the embedding of \(\Oct\) is polarised. 
		It follows from Lemma~\ref{Lem:RTd2} and Theorem~\ref{Thm:PseudoIsoEmbedding}, that the embedding of \(\Oct\) is pseudo-isometric.
	By Theorem~\ref{Thm:StronglyPolarisedEmbedding}, the embedding is also strongly polarised.
		By Corollary~\ref{Cor:RTDual} and Lemma~\ref{Lem:RTovoidal}, the embedding of \(\Oct\) is \dual\ and ovoidal. Therefore, by Theorem~\ref{Thm:Dual+OvoidalEmbedding}, it is also dense.
	

	Assume now that \(\Oct\) is a generalised octagon fully embedded in a metasymplectic space \(\MS\) such that
	\begin{itemize}
		\item[\((S)\)] two points of \(\Oct\) are at distance \(2\) if, and only if, they are symplectic,
		\item[\((P)\)] through every point \(x\) of \(\Oct\) there exists a symp \(x^\polarity\) of \(\MS\) such that all points collinear to \(x\) in \(\Oct\) are contained in \(x^\polarity\),
		\item[\((D)\)] every symp of \(\MS\) contains at least one point of \(\Oct\).
	\end{itemize}
	By Theorem~\ref{Thm:PseudoIsoEmbedding}, the embedding of \(\Oct\) is pseudo-isometric. 
	By Theorem~\ref{Thm:StronglyPolarisedEmbedding}, the embedding of \(\Oct\) is strongly polarised.
	By Theorem~\ref{Thm:Dual+OvoidalEmbedding} the embedding is \dual\ and ovoidal.
	Finally, Lemma~\ref{Cor:KeyCor} completes the proof.
\end{proof}
We now show the two variations.

\begin{proof}[Proof of Theorem~\emph{\ref{Var1}}]
	Let \(\Oct\) be a generalised octagon fully embedded in a metasymplectic space \(\MS\). It suffices to show that if the embedding of \(\Oct\) is polarised (Hypothesis~\((P)\)), then this embedding satisfies Hypothesis~\((S)\) and Hypothesis~\((D)\) if, and only if, 
	\begin{itemize}
		\item[\((D^\star)\)] every symp of \(\MS\) intersects \(\Oct\) in either a point, a line or all points at distance~1 from a point of \(\Oct\).
	\end{itemize}
	Assume first that the embedding of \(\Oct\) satisfies Hypothesis~\((S)\) and Hypothesis~\((D)\). Let \(\Sigma\) be an arbitrary symp of \(\MS\). By Hypothesis~\((D)\), the symp \(\Sigma\) contains at least one point of \(\Oct\). 
	By Theorem~\ref{Thm:PseudoIsoEmbedding}, the embedding is pseudo-isometric. Hence no symp contains points of $\Oct$ at mutual distance at least~3. 
	Assume now that \(\Sigma\) contains two distinct points \(x\) and \(y\) of \(\Oct\) which are not collinear. Then, these points must be symplectic and therefore, by Hypothesis~\((S)\), at distance~2 from each other. Let \(z\) be the unique point of \(\Oct\) collinear to both, then Hypothesis~$(P)$ asserts that  \(\Sigma = \sympThrough{x}{y} = z^\polarity\) and it contains all points at distance~1 from \(z\). Finally, assume $\Sigma$ contains two collinear points $x,y\in\PO$. Then it also contains the line $xy\in\LO$. If it contains an additional point of $\Oct$, we are back in the previous case. This shows Hypothesis~$(D^*)$.
	
	Assume now that the embedding of \(\Oct\) satisfies Hypothesis~\((D^\star)\). Clearly, Hypothesis~\((D)\) is satisfied. 
	It is also easy to see that any pair of points of \(\Oct\) at mutual distance~2 is either collinear or symplectic and that any pair of points of $\Oct$ which are symplectic are at distance~2. It remains to show that no two points of \(\Oct\) at mutual distance~2 are collinear. Suppose, for a contradiction, that the points $x,y\in\Oct$ are collinear in $\MS$ and $d(x,y)=2$. Let $z$ be the unique point of \(\Oct\) collinear to both $x$ and $y$. We consider a symp $\Sigma$ through the line $xy$ of $\MS$ not containing the point $z$. Hypothesis~$(D^*)$ now leads to an obvious contradiction.  
\end{proof}

\begin{proof}[Proof of Theorem~\emph{\ref{Var2}}]
	Let \(\Oct\) be a generalised octagon fully embedded in a metasymplectic space \(\MS\). It suffices to show that if the embedding of \(\Oct\) is polarised and dense, then this embedding satisfies Hypothesis~\((S)\) if, and only if, 
	\begin{itemize}
		\item[\((O)\)] there exist two points of \(\Oct\) that are opposite in \(\MS\),
	\end{itemize}
	Assume first that the polarised and dense embedding of \(\Oct\) satisfies Hypothesis~\((S)\).
	Then, by Theorem~\ref{Thm:PseudoIsoEmbedding}, the embedding is pseudo-isometric and therefore satisfies Hypothesis~\((O)\). 
	
	Assume now that the embedding of \(\Oct\) satisfies Hypothesis~\((O)\).
	By Hypothesis~\((P)\) any two points of \(\Oct\) at distance~2 from each other are either collinear or symplectic.  Hence Lemma~\ref{Lem:SSOpp} asserts that no pair of points of $\Oct$ at mutual distance~$3$ is opposite. Therefore there exist two point of $\Oct$ at mutual distance~$4$ which are opposite in $\MS$. Replacing in the proof of Lemma~\ref{L20} the reference to Lemma~\ref{Lem:d3Special} by a reference to Lemma~\ref{Lem:SSOpp} and the conclusion of the previous sentence, we see that this proof shows that Hypothesis~$(O)$ implies that all points at mutual distance~$4$ of $\Oct$ are opposite in $\MS$. It then follows, again from Lemma~\ref{Lem:SSOpp}, that points of $\Oct$ at mutual distance~$3$ are special. In order to show Hypothesis~$(S)$, we only need to still show that points of $\Oct$ at mutual distance~$2$ are not collinear. Suppose, for a contradiction, that the points $x$ and \(y\) of \(\Oct\) are collinear in $\MS$ and $d(x,y)=2$. Let $z$ be a point of \(\Oct\) with $d(x,z)=4$ and $d(y,z)=2$. Then $x\equiv z$, but $x\perp y$ and either $y\perp z$ (in which case $x$ and $z$ cannot be opposite), or $y\pperp z$, contradicting Lemma~\ref{Lem:SSOpp}. 
\end{proof}

\begin{remark}
	We have shown that an embedding satisfies Hypothesis~\((S)\) if, and only if, the embedding is pseudo-isometric. We conjecture that it is also true that an embedding satisfies Hypothesis~\((O)\) (or the stronger statement that every pair of points of $\Oct$ at mutual distance~$4$ is opposite in $\MS$) if, and only if, the embedding is pseudo-isometric, but this is a significantly harder problem. 
\end{remark}

\paragraph{\bf Acknowledgement}
The authors would like to thank Geertrui Van de Voorde for her helpful suggestions and comments regarding this paper.

\paragraph{\bf Data Availability Statement}
No datasets were generated or analysed during this study.

\paragraph{\bf Conflicts of interest}
The authors have no conflicts of interest to declare.

\end{document}